\documentclass[final,1p,times]{elsarticle}
\usepackage{amsmath,amssymb,amsfonts,amsthm,graphicx, float, lineno}
\usepackage{hyperref}
\newtheorem{theorem}{Theorem}[section]

\newtheorem{corollary}[theorem]{Corollary}
\newtheorem{lemma}[theorem]{Lemma}

\newtheorem{definition}[theorem]{Definition}

\numberwithin{equation}{section}

\begin{document}

\begin{frontmatter}

\title{On the Extreme Value Behavior of $\vartheta$-Expansions}

\author[upb,ismma]{Gabriela Ileana Sebe}
\ead{igsebe@yahoo.com}
\address[upb]{Faculty of Applied Sciences, 
National University of Science and Technology  
POLITEHNICA Bucharest, \\ Splaiul Independentei 313, 060042, Bucharest, Romania}
\address[ismma]{Gheorghe Mihoc-Caius Iacob Institute of Mathematical Statistics and Applied Mathematics of the Romanian Academy, Calea 13 Sept. 13, 050711 Bucharest, Romania}

\author[anmb]{Dan Lascu}
\ead{lascudan@gmail.com}
\address[anmb]{Romanian Naval Academy ``Mircea cel Batran", 1 Fulgerului, 900218 Constanta, Romania}

\author[um]{Bilel Selmi\corref{cor1}}
\ead{bilel.selmi@fsm.rnu.tn, bilel.selmi@isetgb.rnu.tn}
\address[um]{Analysis, Probability and Fractals Laboratory LR18ES17, Department of Mathematics, Faculty of Sciences of Monastir, University of Monastir, 5000 Monastir, Tunisia}
\cortext[cor1]{Corresponding author}

\begin{abstract}
The main objective of this paper is to develop extreme value theory for $\vartheta$-expansions. We establish the limit distribution of the maximum value in a $\vartheta$-continued fraction mixing stationary stochastic process, along with some related results. These findings are analogous to the theorems of J. Galambos and W. Philipp for regular continued fractions. Additionally, we emphasize that a Borel-Bernstein type theorem plays a crucial role.
\end{abstract}

\end{frontmatter}

\sloppy

\section{Introduction}

The study initiated by Bhattacharya and Goswami \cite{BG-2000} in the context of random number generations led to the concept of continued fraction expansion of a number in terms of an irrational $\vartheta \in (0,1)$.  
This novel expansion of positive real numbers, distinct from the regular continued fraction (RCF) expansion, is referred to as the \textit{$\vartheta$-expansion}.  
We note that the special case $\vartheta = 1$ corresponds to the RCF expansion. 

For a fixed $\vartheta \in (0, 1)$, Chakraborty and Rao \cite{CR-2003} introduced a generalization of the Gauss map, defined as $T_{\vartheta}: [0,\vartheta] \to [0,\vartheta]$:
\begin{equation}
T_{\vartheta}(x):=
\left\{
\begin{array}{ll}
{\displaystyle \frac{1}{x} - \vartheta \left \lfloor \frac{1}{x \vartheta} \right\rfloor}, &
{\displaystyle \text{if } x \in (0, \vartheta],}\\
\\
0, & \text{if } x=0.
\end{array}
\right. \label{1.1}
\end{equation}
Here, $\left\lfloor \cdot \right\rfloor$ denotes the integer part.  
Every $x \in \left(0, \vartheta \right)$ can then be expressed as a finite or infinite $\vartheta$-expansion:
\begin{equation}
x = \frac{1}{\displaystyle b_1\vartheta
+\frac{1}{\displaystyle b_2\vartheta
+ \frac{1}{\displaystyle b_3\vartheta + \ddots} }} =: [b_1 \vartheta, b_2 \vartheta, b_3 \vartheta, \ldots]. \label{1.2}
\end{equation}
The sequence $\{b_n\}$ is determined by
\begin{equation*}
b_1=b_1(x) := \left\{\begin{array}{ll}
\left\lfloor \displaystyle \frac{1}{x \vartheta}\right\rfloor,  & \text{if }  x \neq 0, \\
\\
\infty, & \text{if }  x = 0,
\end{array} \right.
\end{equation*}
and for $n \geq 2$,
\begin{equation*}
b_n=b_n(x) := b_1\left(T_{\vartheta}^{n-1}(x)\right), \quad n \in \mathbb{N}_+ := \left\{1, 2, 3, \ldots\right\},
\end{equation*}
with the convention that $T_{\vartheta}^0 (x) = x$.  
The positive integers $b_n \in \mathbb{N}_+$ are called the \textit{digits} or \textit{partial quotients} of $x$.  

Let $\mathcal{B}_{[0,\vartheta]}$ be the $\sigma$-algebra of all Borel subsets of $[0,\vartheta]$. It is evident that the digits $b_n$, $n \in \mathbb{N}_+$, are random variables defined almost surely on $\left( [0, \vartheta], \mathcal{B}_{[0,\vartheta]}\right)$ with respect to any probability measure on $\mathcal{B}_{[0,\vartheta]}$ that assigns probability $0$ to the set of rationals in $[0, \vartheta]$.  
An example of such a measure is the Lebesgue measure $\lambda_{\vartheta}$ on $[0, \vartheta]$.
It was shown in \cite{CR-2003,SL-2014} that this expansion exhibits many of the standard properties of regular continued fractions (RCFs).  
A natural question arises: does the dynamical system defined by the transformation $T_{\vartheta}$ admit an absolutely continuous invariant probability measure, similar to the Gauss measure in the case $\vartheta = 1$?  Chakraborty and Rao \cite{CR-2003} identified that for certain values of $\vartheta$ (for instance, when $\vartheta^2 = \frac{1}{m}$, where $m \in \mathbb{N}_+$), the invariant measure for the transformation $T_{\vartheta}$ is given by  
\begin{equation}\label{1.3}
\mathrm{d}\gamma_{\vartheta} := \frac{1}{\log \left(1+\vartheta^{2}\right)}
\frac{\vartheta \,\mathrm{d}x}{1 + \vartheta x}.
\end{equation}

Moreover, when $\vartheta^2 = \frac{1}{m}$, $m \in \mathbb{N}_+$, a number $x \in (0, \vartheta)$ has the $\vartheta$-expansion $[b_1 \vartheta, b_2 \vartheta, b_3 \vartheta, \ldots]$ if and only if the following conditions hold:  
\begin{enumerate}
\item [(i)] $b_n \geq m$ for all $n \in \mathbb{N}_+$.
\item [(ii)] If $x$ has a finite expansion, i.e., $x = [b_1 \vartheta, b_2 \vartheta, \ldots,  b_n \vartheta]$, then $b_n \geq m+1$.
\end{enumerate}
It was further established in \cite{CR-2003} that the dynamical system  
$([0,\vartheta], T_{\vartheta})$ is ergodic and that the measure $\gamma_{\vartheta}$ remains invariant under $T_{\vartheta}$, meaning  
\[
\gamma_{\vartheta} (A) = \gamma_{\vartheta} (T^{-1}_{\vartheta}(A))
\]
for all $A \in {\mathcal{B}}_{[0, \vartheta]}$.  
Consequently, the sequence $(b_n)_{n \in \mathbb{N}_+}$ forms a strictly stationary process on $([0,\vartheta],{\mathcal B}_{[0,\vartheta]},\gamma_{\vartheta})$.  
For further results on $\vartheta$-expansions, see \cite{CD-2004,Sebe-2017,SL-2014, SL-2019, SL-2025} and the references therein.

Every irrational number $x \in (0, \vartheta) \setminus \mathbb{Q}=: \Omega$ possesses an infinite $\vartheta$-expansion. It is important to note that for all $n \in \mathbb{N}_+$, we have $b_n(x) \geq m$, and the transformation satisfies  
\[
T^n_{\vartheta}([b_1 \vartheta, b_2 \vartheta, \ldots]) = [b_{n+1} \vartheta, b_{n+2} \vartheta, \ldots].
\]
For all $n \in \mathbb{N}_+$, the finite truncation of (\ref{1.2}) is given by  
\[
\frac{p_n(x)}{q_n(x)} = [b_1(x) \vartheta, b_2(x) \vartheta, \ldots, b_n(x) \vartheta],
\]
which is referred to as the $n$-\textit{th convergent} of the $\vartheta$-expansion of $x$.  
For every infinite $\vartheta$-expansion $[b_1 \vartheta, b_2 \vartheta, \ldots]$, the sequences $\{p_n\}_{n \geq -1}$ and $\{q_n\}_{n \geq -1}$ satisfy the recursive relations  
\begin{eqnarray}
p_n(x) &=& b_n(x) \vartheta p_{n-1}(x) + p_{n-2}(x), \quad \label{1.4} \\
q_n(x) &=& b_n(x) \vartheta q_{n-1}(x) + q_{n-2}(x), \quad \label{1.5}
\end{eqnarray}
with initial conditions $p_{-1}(x) := 1$, $p_0(x) := 0$, $q_{-1}(x) := 0$, and $q_{0}(x) := 1$.  
By induction, it follows that  
\begin{equation}
p_{n-1}(x)q_{n}(x) - p_{n}(x)q_{n-1}(x) = (-1)^{n}, \quad n \in \mathbb{N}. \label{1.6}
\end{equation}
Using (\ref{1.4}) and (\ref{1.5}), we can derive the relation  
\begin{equation}
x = \frac{p_n(x) + T^n_{\vartheta}(x)p_{n-1}(x)}
{q_n(x) + T^n_{\vartheta}(x)q_{n-1}(x)}, \quad n \geq 1. \label{1.7}
\end{equation}
From (\ref{1.6}) and (\ref{1.7}), we obtain  
\begin{equation}
\left| x - \frac{p_n(x)}{q_n(x)} \right|  = \frac{1}{q_n(x)\left(\left(T^n_{\vartheta}(x)\right)^{-1} q_n(x)+q_{n-1}(x) \right)}, \quad n \geq 1. \label{1.8}
\end{equation}
Since $b_{n+1}(x)\vartheta \leq \left(T^n_{\vartheta}(x)\right)^{-1} \leq (b_{n+1}(x)+1)\vartheta$, substituting (\ref{1.4}) and (\ref{1.5}) into (\ref{1.8}) yields  
\begin{equation}
\frac{1}{q_n(x)(q_{n+1}(x)+\vartheta q_{n}(x))} \leq  \left| x - \frac{p_n(x)}{q_n(x)} \right| \leq \frac{1}{q_n(x)q_{n+1}(x)}, \quad n \geq 1. \label{1.9}
\end{equation}
From (\ref{1.5}), we conclude that $q_n(x) \geq \vartheta$ for all $n \in \mathbb{N}_+$. Moreover, applying induction to (\ref{1.5}) gives  
\begin{equation}
q_n(x) \geq  \left\lfloor \frac{n}{2} \right \rfloor \vartheta^2. \label{1.10} 
\end{equation}
Finally, combining (\ref{1.9}) and (\ref{1.10}), we deduce that  
\[
[b_1(x) \vartheta, b_2(x) \vartheta, \ldots, b_n(x) \vartheta] \to x \quad \text{as} \quad n \to \infty.
\]
Relation (\ref{1.9}) implies that the accuracy of this approximation depends on the growth rate of the partial quotients.  

In the case of regular continued fractions (RCFs), Borel \cite{Borel} and Bernstein \cite{Bern} established a fundamental result known as the \textit{Borel-Bernstein theorem} or the \textit{``$0$-$1$'' law}, which describes the growth rate of partial quotients in terms of Lebesgue measure.  
Our first result, Theorem \ref{th.B-B}, presents an analogous version of the Borel-Bernstein theorem for $\vartheta$-expansions.  
Furthermore, in Section 5, we demonstrate that this type of theorem plays a crucial role in the study of $\vartheta$-expansions.  
Sections 4 and 5 contain new results concerning extreme value theory for $\vartheta$-expansions. These results are novel in the sense that they do not appear to have been previously stated elsewhere.  
Extreme value theory for RCF digits first emerged in the 1970s. The foundational results of Galambos \cite{G-1972, G-1973, G-1974} regarding the maximal RCF digit were later refined by Philipp \cite{Ph-1976}, providing a complete resolution to a conjecture posed by Erd\H{o}s.   In Section 4, we establish a Fr\'echet law concerning the partial maxima of the growth rate of the digit sequence.  
Theorems \ref{Th.4.4} and \ref{Th.4.5} extend the prior work of Galambos \cite{G-1972} and Philipp \cite{Ph-1976} on the asymptotic behavior of the largest digit in RCF-expansions.  To obtain these results, we rely on the $\vartheta$-continued fraction mixing property, along with a specific condition on the rate of convergence of mixing.  
In Section 5, we present iterated logarithm results (Theorem \ref{th.5.2} and Corollary \ref{cor.5.4}) concerning the largest digit in $\vartheta$-expansions.  

\section{Preliminaries} 

Let us fix $\vartheta^2 = 1/m$, where $m \in \mathbb{N}_+$. Define  
\[
\mathbb{N}_m := \{m, m+1, \ldots\},
\]
so that the partial quotients $b_n$, for $n \in \mathbb{N}_+$, take positive integer values in $\mathbb{N}_m$.  

We now introduce a natural partition of the interval $[0, \vartheta]$ that corresponds to $\vartheta$-expansions. This partition is generated by cylinders of rank $n$. For any $n \in \mathbb{N}_+$ and  
$i^{(n)}=(i_1, \ldots, i_n) \in \mathbb{N}_m^n$, we define the {\it $n$-th cylinder} of the $\vartheta$-expansion by  
\begin{equation*}
C \left(i^{(n)}\right) = \{x \in \Omega:  b_k(x) = i_k \text{ for } k=1, \ldots, n \},
\end{equation*}
where $C \left(i^{(0)}\right) = [0, \vartheta]$.  
For any $i \in \mathbb{N}_m$, we have  
\begin{equation} \label{2.01}
C(i) = \left\{x \in \Omega: b_1(x) = i \right\} = \left( \frac{1}{(i+1)\vartheta}, \frac{1}{i \vartheta} \right).
\end{equation}
For $n \in \mathbb{N}_+$ and $i_n \in \mathbb{N}_m$, we use the notation  
\[
C(b_1, \ldots, b_n) = C \left(i^{(n)}\right).
\]  

Next, we recall some known results for later use.  
From the definition of $T_{\vartheta}$ and (\ref{1.7}), for any $n \in \mathbb{N}_+$ and $(b_1, \ldots, b_n) \in \mathbb{N}_m^n$, we obtain  
\begin{equation}
C(b_1, \ldots, b_n) = \left\{
\begin{array}{ll}
	\left[ \displaystyle \frac{p_n}{q_n}, \frac{p_n+ \vartheta p_{n-1}}{q_n+ \vartheta q_{n-1}} \right)  & \text{if $n$ is even}, \\
	\\
	\left(\displaystyle \frac{p_n+ \vartheta p_{n-1}}{q_n+ \vartheta q_{n-1}}, \frac{p_n}{q_n} \right] & \text{if $n$ is odd}.
\end{array}
\right. \label{2.1}
\end{equation}

Using (\ref{1.6}), we obtain  
\begin{eqnarray}
\lambda_{\vartheta}\left(C\left(b_1, \ldots, b_n\right)\right) &=& \frac{1}{\vartheta} \left| \frac{p_n}{q_n} - \frac{p_n+ \vartheta p_{n-1}}{q_n+\vartheta q_{n-1}} \right| \notag \\
&=& \frac{1}{q_n (q_n + \vartheta q_{n-1})} = \frac{1}{q^2_n (1 + \vartheta s_{n})}, \label{2.2}
\end{eqnarray}
where $s_n = \frac{q_{n-1}}{q_n}$ for $n \in \mathbb{N}_+$, with $s_0=0$.  
Since $s_n \in [0, \vartheta]$, it follows from (\ref{2.2}) that  
\begin{equation}
\frac{1}{2q^2_n} \leq \frac{1}{(1+\vartheta^2) q^2_n} \leq \lambda_{\vartheta}\left(C\left(b_1, \ldots, b_n\right)\right) \leq \frac{1}{q^2_n}. \label{2.3}
\end{equation}

An interesting problem is to compute the approximate proportion of the $n$-th level cylinder set $C\left(b_1, \ldots, b_n\right)$ that is occupied by each of the $(n+1)$-th level cylinder sets $C\left(b_1, \ldots, b_n, k\right)$:  
The endpoints of the interval $C\left(b_1, \ldots, b_n, k\right)$ are given by  
\[
\frac{p_{n+1}}{q_{n+1}}, \quad \frac{p_{n+1}+ \vartheta p_n}{q_{n+1}+\vartheta q_n},
\]
where  
\[
p_{n+1} = k \vartheta p_n + p_{n-1}, \quad q_{n+1} = k \vartheta q_n + q_{n-1}.
\]
Thus, we have  
\begin{equation*}
\frac{p_{n+1}}{q_{n+1}} = \frac{k \vartheta p_n + p_{n-1}}{k \vartheta q_n + q_{n-1}}, \quad
\frac{p_{n+1}+\vartheta p_n}{q_{n+1}+\vartheta q_n} = \frac{(k+1) \vartheta p_n + p_{n-1}}{(k+1) \vartheta q_n + q_{n-1}}.
\end{equation*}

A direct computation yields  
\begin{equation*}
\lambda_{\vartheta}\left(C\left(b_1, \ldots, b_n, k\right)\right) = \frac{1}{(k\vartheta q_n + q_{n-1})((k+1)\vartheta q_n + q_{n-1})} =  \frac{1}{k^2q^2_n\left(\vartheta + \frac{s_n}{k}\right)\left( \left( 1+\frac{1}{k} \right) \vartheta +\frac{s_n}{k} \right) }. 
\end{equation*}
From (\ref{2.2}), it follows that  
\begin{eqnarray*}
\frac{\lambda_{\vartheta}\left(C\left(b_1, \ldots, b_n, k\right)\right)}{\lambda_{\vartheta}\left(C\left(b_1, \ldots, b_n\right)\right)} 
&=& 
\frac{q^2_n(1+\vartheta s_n)}{k^2q^2_n\left(\vartheta + \frac{s_n}{k}\right)\left( \left( 1+\frac{1}{k} \right) \vartheta +\frac{s_n}{k} \right)} \\
&=& \frac{1+\vartheta s_n}{k^2 \left(\vartheta + \frac{s_n}{k}\right)\left( \left( 1+\frac{1}{k} \right) \vartheta +\frac{s_n}{k} \right)}. 
\end{eqnarray*}
Since  
\[
\vartheta^2 < \left(\vartheta + \frac{s_n}{k}\right)\left( \left( 1+\frac{1}{k} \right) \vartheta +\frac{s_n}{k} \right) 
< \vartheta^2 \left( 1+\frac{1}{k} \right) \left( 1+\frac{2}{k} \right) < 6 \vartheta^2 < 6,  
\]
for $k \geq m$, we obtain  
\begin{equation}
\frac{1}{6k^2} < \frac{\lambda_{\vartheta}\left(C\left(b_1, \ldots, b_n, k\right)\right)}{\lambda_{\vartheta}\left(C\left(b_1, \ldots, b_n\right)\right)} < \frac{m+1}{k^2}.  \label{2.4}
\end{equation}

Next, we state some lemmas for later use.
\begin{lemma} \label{lema2.1}
Let $k \geq m$, then 
\[
\frac{1}{6k^2} < \lambda_{\vartheta}\left( \bigcup_{b_1,\ldots,b_n \geq m} C\left(b_1, \ldots, b_n, k\right)\right) < \frac{m+1}{k^2}. 
\]
\end{lemma}
\begin{proof}
Using (\ref{2.4}) and the fact that  
\[
\sum_{b_1,\ldots,b_n \geq m} \lambda_{\vartheta}\left( C\left(b_1, \ldots, b_n \right)\right) = 1,
\]  
the proof is complete.
\end{proof}

We present the following widely known and highly useful result.
\begin{lemma}[Borel-Cantelli] \label{lema2.2}
Let $(X,{\cal X}, \mu)$ be a measurable space. Let $\{C_n\}_{n\geq 1}$ be a sequence of ${\cal X}$-measurable sets and define the $\limsup$ set
\[
C_{\infty} = \limsup_{n \to \infty} C_n = \bigcap_{n\geq 1} \bigcup_{m\geq n} C_m = \{x \in X: x \in C_n \mbox{ for infinitely many } n \in \mathbb{N}_+ \}. 
\]
Then, if $\displaystyle \sum_{n \geq 1} \mu(C_n) < \infty$, we have that $\mu(C_{\infty})=0$. 
\end{lemma}

\section{A Borel-Bernstein-type theorem}
Our primary result is stated in the following theorem:

\begin{theorem} [Borel-Bernstein-type theorem] \label{th.B-B}
Let $\varphi : \mathbb{N}_+ \to (0, +\infty)$ be a function, and define
\[
A_{\varphi} = \{x \in \Omega: b_n(x) > \varphi(n) \text{ for infinitely many } n \in \mathbb{N}_+ \}.
\]
Then, we have
\[
\lambda_{\vartheta} (A_{\varphi}) =
\begin{cases}
0, & \text{if } \sum_{n \geq 1} \frac{1}{\varphi(n)} < \infty, \\
1, & \text{if } \sum_{n \geq 1} \frac{1}{\varphi(n)} = \infty.
\end{cases}
\]
\end{theorem}

\begin{proof}
Define $A_n = \{x \in \Omega: b_n(x) > \varphi(n)\}$, so that $A_{\varphi} = \limsup_{n \to \infty} A_n = \bigcap_{j \geq 1} \bigcup_{n \geq j} A_n$. 
By Lemma \ref{lema2.1}, we obtain
\[
\lambda_{\vartheta}(A_n) = \lambda_{\vartheta}\left( \bigcup_{b_1,\ldots,b_{n-1} \geq m} \bigcup_{k>\varphi(n)} C\left(b_1, \ldots, b_{n-1}, k\right)\right)
< \sum_{k \geq \lfloor \varphi(n) \rfloor+1} \frac{m+1}{k^2} < \frac{m+1}{\lfloor \varphi(n) \rfloor} < \frac{2(m+1)}{\varphi(n)}.
\]
By the Borel-Cantelli lemma, it follows that $\lambda_{\vartheta} (A_{\varphi}) = 0$ when $\sum_{n \geq 1} \frac{1}{\varphi(n)} < \infty$.
Now, suppose $\sum_{n \geq 1} \frac{1}{\varphi(n)} = \infty$. We observe that
\[
\lambda_{\vartheta} (A_{\varphi}) = \lambda_{\vartheta} \left( \bigcap_{j\geq 1} \bigcup_{n \geq j} A_n\right)  =1 \iff 
\lambda_{\vartheta} (A^c_{\varphi}) = \lambda_{\vartheta} \left( \bigcup_{j\geq 1} \bigcap_{n \geq j} A^c_n\right) = 0.
\]
Thus, it suffices to prove $\lambda_{\vartheta} \left(\bigcap_{n \geq j} A^c_n \right)=0$, where
$A^c_n = \{x \in \Omega: b_n(x) \leq \varphi (n) \}$.
Define
\[
B_{j,\ell} = \bigcap_{j<n \leq j+\ell} A^c_n.
\]
Then,
\[
\lambda_{\vartheta}\left(\bigcap_{n \geq j+1} A^c_n \right) = \lim_{\ell \to \infty} \lambda_{\vartheta} (B_{j,\ell}).
\]
By the definition of $B_{j, \ell}$, we have
\[
B_{j,1} = \bigcup_{b_1,\ldots,b_{j} \geq m} \bigcup_{m \leq k \leq \varphi(j+1)} C\left(b_1, \ldots, b_{j}, k\right).
\]
Applying Lemma \ref{lema2.1}, we obtain
\[
\sum_{k \geq \lfloor \varphi(j+1)\rfloor+1} \frac{\lambda_{\vartheta}\left(C\left(b_1, \ldots, b_j, k\right)\right)}{\lambda_{\vartheta}\left(C\left(b_1, \ldots, b_j\right)\right)} > 
\sum_{k \geq \lfloor \varphi(j+1)\rfloor+1} \frac{1}{6k^2} > \frac{1}{6(\varphi(j+1)+1)}.
\]
Thus,
\[
\lambda_{\vartheta} (B_{j,1}) \leq 1 - \frac{1}{6(\varphi(j+1)+1)}.
\]
Since
\[
B_{j, \ell+1} = B_{j, \ell} \cap \{x \in \Omega: b_{j+\ell+1} \leq \varphi (j+\ell+1)\},
\]
we obtain by induction,
\[
\lambda_{\vartheta} (B_{j,\ell}) \leq \prod^{\ell}_{i=1} \left( 1 - \frac{1}{12\varphi(j+i)}\right) 
\leq \exp\left(-\sum_{i=1}^{\ell} \frac{1}{12\varphi(j+i)}\right).
\]
Since $\sum_{n \geq 1} \frac{1}{\varphi(n)} = \infty$, we conclude that $\lim_{\ell \to \infty} \lambda_{\vartheta} (B_{j,\ell})=0$. 
Therefore, $\lambda_{\vartheta} (A^c_{\varphi}) = 0$, completing the proof.
\end{proof}

\begin{corollary} \label{cor.3.2}
For $\lambda_{\vartheta}$-almost every $x \in [0, \vartheta]$, it holds that 
\[
b_n(x) > n \log n \quad \text{for infinitely many} \quad n \in \mathbb{N}_+,
\]
while for each $\varepsilon > 0$, we have 
\[
b_n(x) < n (\log n)^{1 + \varepsilon} \quad \text{for all sufficiently large} \quad n \in \mathbb{N}_+.
\]
\end{corollary}
\begin{proof}
This result follows directly from Theorem \ref{th.B-B} based on the following observation 
\[
\sum_{n \geq 1} \frac{1}{n \log n} = \infty \quad \text{and} \quad \sum_{n \geq 1} \frac{1}{n (\log n)^{1 + \varepsilon}} < \infty
\]
for all $\varepsilon > 0$.
\end{proof}

\section{The asymptotic behavior of the largest digit in $\vartheta$-expansions}

In the following, we will rely on the fundamental principles of the metric theory of $\vartheta$-expansions. One such principle is that the stochastic process derived from the digits of the $\vartheta$-expansion exhibits the $\psi$-mixing property.

\begin{definition}[$\psi$-mixing] \label{def.4.1}
Let $(X, \mathcal{X}, \mu)$ be a probability space, and let $\xi_j : X \to \mathbb{R}$ be a stationary sequence of random variables. 
For any $k \in \mathbb{N}_+$, let $\mathcal{B}_1^k = \sigma(\xi_1, \dots, \xi_k)$ and 
$\mathcal{B}_k^\infty = \sigma(\xi_k, \xi_{k+1}, \dots)$ denote the $\sigma$-algebras generated by the random variables $\xi_1, \dots, \xi_k$ and $\xi_k, \xi_{k+1}, \dots$, respectively.
The sequence $\{\xi_j\}$ is said to be \textit{$\psi$-mixing} if for any sets $A \in \mathcal{B}_1^k$ and $B \in \mathcal{B}_{k+n}^\infty$, the following inequality holds:
\[
\left| \mu(A \cap B) - \mu(A) \mu(B) \right| \leq \psi(n) \mu(A) \mu(B),
\]
where $\psi : \mathbb{N}_+ \to \mathbb{R}$ is a function satisfying $\psi(n) \to 0$ as $n \to \infty$.
\end{definition}
The random variables $b_n(x)$, for $n \in \mathbb{N}_+$, form a stationary sequence due to the invariance of the measure $\gamma_{\vartheta}$ with respect to the transformation $T_{\vartheta}$.

\begin{lemma} \label{lema4.2}
For all \( n \in \mathbb{N}_+ \) and \( w \in \mathbb{N}_m \),
\[
\gamma_{\vartheta}(b_n(x) \geq w) = \frac{1}{\log(1+\vartheta^2)} \log\left( 1 + \frac{1}{w} \right) =: p_{\vartheta}(w).
\]
\end{lemma}

\begin{proof}
Using (\ref{2.01}) and the fact that \( (b_n)_{n \in \mathbb{N}_+} \) is a strictly stationary sequence, as the transformation \( T_{\vartheta} \) is measure-preserving with respect to \( \gamma_{\vartheta} \), we have
\[
\gamma_{\vartheta}(b_n(x) \geq w) = \gamma_{\vartheta}(b_1(x) \geq w) = \frac{1}{\log(1+\vartheta^2)} \sum_{k \geq w} \int_{\frac{1}{k \vartheta}}^{\frac{1}{(k+1) \vartheta}} \frac{\vartheta \, \mathrm{d}x}{1+\vartheta x}.
\]
This simplifies to
\[
\frac{1}{\log(1+\vartheta^2)} \sum_{k \geq w} \left( \log\left( 1 + \frac{1}{k} \right) - \log\left( 1 + \frac{1}{k+1} \right) \right),
\]
which results in
\[
\frac{1}{\log(1+\vartheta^2)} \log\left( 1 + \frac{1}{w} \right).
\]
\end{proof}

\begin{lemma} \label{lema4.2a}
Let \( \{ f_{\vartheta,n}(x) \}_{n \geq 1} \) be a sequence of functions \( f_{\vartheta,n} \in C^2[0, \vartheta] \) defined recursively by
\[
f_{\vartheta,n+1}(x) = \sum_{i \geq m} \left( f_{\vartheta,n} \left( \frac{1}{\vartheta i} \right) - f_{\vartheta,n} \left( \frac{1}{x + \vartheta i} \right) \right), \quad n \in \mathbb{N},
\]
with \( f_{\vartheta, 0}(0) = 0 \) and \( f_{\vartheta, 0}(\vartheta) = 1 \).

Define
\begin{equation} 
g_{\vartheta, n}(x) = (\vartheta x + 1) f'_{\vartheta, n}(x), \quad x \in [0, \vartheta]. \label{4.001}
\end{equation}
Then
\begin{equation} \label{4.002}
\left\| g'_{\vartheta, n} \right\| \leq q^n_{\vartheta} \cdot \left\| g'_{\vartheta, 0} \right\|, \quad n \in \mathbb{N}_+,
\end{equation}
where
\begin{equation} \label{4.003}
q_{\vartheta} := m \left( \sum_{i \geq m} \left( \frac{m}{i^3(i+1)} + \frac{i+1-m}{i(i+1)^3} \right) \right) < 1.
\end{equation}
Here, \( \| \cdot \| \) denotes the supremum norm.
\end{lemma}

\begin{proof}
Since
\[
g_{\vartheta,n+1}(x) = \sum_{i \geq m} P_{\vartheta,i}(x) g_{\vartheta,n}\left( u_{\vartheta,i}(x) \right),
\]
where
\[
P_{\vartheta,i}(x) := \frac{\vartheta x + 1}{(x + \vartheta i)(x + \vartheta (i+1))} = \frac{1}{\vartheta} \left[ \frac{1 - \vartheta^2 i}{x + \vartheta i} - \frac{1 - \vartheta^2 (i+1)}{x + \vartheta (i+1)} \right]
\]
and
\[
u_{\vartheta,i}(x) := \frac{1}{x + \vartheta i},
\]
we have
\[
g'_{\vartheta,n+1}(x) = \sum_{i \geq m} \frac{1 - \vartheta^2 (i+1)}{(x + \vartheta i)(x + \vartheta (i+1))^3} f'_{\vartheta,n}(\alpha_{\vartheta,i}) - \sum_{i \geq m} \frac{P_{\vartheta,i}(x)}{(x + \vartheta i)^2} f'_{\vartheta,n}(u_{\vartheta,i}(x)),
\]
where \( u_{\vartheta,i+1}(x) < \alpha_{\vartheta,i} < u_{\vartheta,i}(x) \).
Thus,
\[
\left\| g'_{\vartheta, n+1} \right\| \leq \left\| g'_{\vartheta, n} \right\| \cdot \max_{x \in [0, \vartheta]} \left( \sum_{i \geq m} \frac{\vartheta^2 (i+1) - 1}{(x + \vartheta i)(x + \vartheta (i+1))^3} + \sum_{i \geq m} \frac{P_{\vartheta,i}(x)}{(x + \vartheta i)^2} \right).
\]
Using the inequalities
\[
\frac{\vartheta^2 (i+1) - 1}{(x + \vartheta i)(x + \vartheta (i+1))^3} \leq m^2 \frac{\vartheta^2 (i+1) - 1}{i(i+1)^3}
\]
and
\[
\sum_{i \geq m} \frac{P_{\vartheta,i}(x)}{(x + \vartheta i)^2} \leq m^2 \sum_{i \geq m} \frac{1}{i^3(i+1)},
\]
we obtain (\ref{4.002}) and (\ref{4.003}).
\end{proof}

The random variables \( b_n(x) \), \( n \in \mathbb{N}_+ \), are not independent. However, they satisfy a \( \psi \)-mixing condition. In fact, the sequence \( \{b_n\}_{n \in \mathbb{N}_+} \) is \( \psi \)-mixing under \( \gamma_{\vartheta} \), and the function \( \psi \) decays at an exponential rate:

\begin{lemma} \label{lema4.3}
For any sets \( A \in \mathcal{B}_1^k = \sigma(b_1, \ldots, b_k) \) and \( B \in \mathcal{B}_{k+n}^{\infty} = \sigma(b_{k+n}, b_{k+n+1}, \ldots) \), we have
\begin{equation} \label{4.004} 
\left| \gamma_{\vartheta}(A \cap B) - \gamma_{\vartheta}(A) \gamma_{\vartheta}(B) \right| \leq K_{\vartheta} q^n_{\vartheta} \gamma_{\vartheta}(A) \gamma_{\vartheta}(B),
\end{equation}
where \( 0 < q_{\vartheta} < 1 \) and \( K_{\vartheta} \) is a positive constant.
\end{lemma}

\begin{proof}
Let \( C_k \) be the \( k \)-th cylinder with endpoints \( \frac{p_k}{q_k} \) and \( \frac{p_k + \vartheta p_{k-1}}{q_k + \vartheta q_{k-1}} \). Define
\[ f_{\vartheta, n}(x) = \gamma_{\vartheta} \left( T^{n+k}_{\vartheta}(\omega) < x \mid C_k \right) := \frac{\gamma_{\vartheta} \left( \left( T^{n+k}_{\vartheta}(\omega) < x \right) \cap C_k \right)}{\gamma_{\vartheta}(C_k)}, \]
which is the conditional distribution function of \( T^{n+k}_{\vartheta}(\omega) \) given \( C_k \).

It is evident that \( \left( \left( T^k_{\vartheta}(\omega) < x \right) \cap C_k \right) \) is an interval with endpoints \( \frac{p_k}{q_k} \) and \( \frac{p_k + x \vartheta p_{k-1}}{q_k + x \vartheta q_{k-1}} \). Thus, we obtain
\[
f_{\vartheta, 0}(x) = \frac{1}{\gamma_{\vartheta}(C_k)} \frac{(-1)^k}{\log\left(1 + \vartheta^2\right)} \left( \log\left( 1 + \vartheta \frac{p_k + x \vartheta p_{k-1}}{q_k + x \vartheta q_{k-1}} \right) - \log\left( 1 + \vartheta \frac{p_k}{q_k} \right) \right).
\]
If \( g_{\vartheta, n} \) is defined as in Lemma \ref{lema4.2a}, we define
\[
K_{\vartheta} := \sup_{x \in [0, \vartheta]} \left| g'_{\vartheta, 0}(x) \right| = \| g'_{\vartheta, 0} \|.
\]
Hence, by (\ref{4.001}) and (\ref{4.002}),
\begin{equation} \label{4.005}
\left| f'_{\vartheta, n}(x) - \frac{\vartheta}{(\vartheta x + 1)\log(1 + \vartheta^2)} \right| \leq \frac{\left| g_{\vartheta, n}(x) - g_{\vartheta, n}(0) \right|}{\vartheta x + 1} + \frac{\left| g_{\vartheta, n}(0) - \frac{\vartheta}{\log(1 + \vartheta^2)} \right|}{\vartheta x + 1}
\end{equation}
and
\[
\left| g_{\vartheta, n}(x) - g_{\vartheta, n}(0) \right| = \left| \int_0^x g'_{\vartheta, n}(t) \, \mathrm{d}t \right| \leq \| g'_{\vartheta, n} \| \cdot x \leq K_{\vartheta} q^n_{\vartheta} x.
\]
Also, for some \( 0 < v_{\vartheta} < 1 \),
\[
1 = f_{\vartheta, n}(\vartheta) = \int_0^{\vartheta} f'_{\vartheta, n}(t) \, \mathrm{d}t = \int_0^{\vartheta} \frac{g_{\vartheta,n}(t)}{\vartheta t + 1} \, \mathrm{d}t
\]
\[
= g_{\vartheta, n}(0) \frac{\log(1 + \vartheta^2)}{\vartheta} + \int_0^{\vartheta} \frac{g_{\vartheta, n}(t) - g_{\vartheta, n}(0)}{\vartheta t + 1} \, \mathrm{d}t
\]
\[
= g_{\vartheta, n}(0) \frac{\log(1 + \vartheta^2)}{\vartheta} + v_{\vartheta} K_{\vartheta} q^n_{\vartheta} \left( 1 - \frac{\log(1 + \vartheta^2)}{\vartheta^2} \right),
\]
so
\[
g_{\vartheta, n}(0) = \frac{\vartheta}{\log(1 + \vartheta^2)} + v_{\vartheta} \frac{K_{\vartheta} q^n_{\vartheta}}{\vartheta} \left( 1 - \frac{\vartheta^2}{\log(1 + \vartheta^2)} \right).
\]
Thus, from (\ref{4.005}),

\begin{eqnarray}
\left| f'_{\vartheta, n}(x) - \frac{\vartheta}{(\vartheta x + 1)\log(1 + \vartheta^2)} \right| &\leq& \frac{K_{\vartheta} q^n_{\vartheta}}{\vartheta x + 1} + \frac{K_{\vartheta} q^n_{\vartheta}}{\vartheta} \left( \frac{\vartheta^2}{\log(1 + \vartheta^2)} - 1 \right) \frac{1}{\vartheta x + 1} \nonumber \\ 
&<& \frac{K_{\vartheta} q^n_{\vartheta}}{\vartheta x + 1} \frac{\vartheta}{\log(1 + \vartheta^2)}. \label{4.006}
\end{eqnarray}

Integrating (\ref{4.006}) over \( F \), we obtain
\[
\left| \gamma_{\vartheta} \left( T^{-(n+k)}_{\vartheta}(F) \mid C_k \right) - \gamma_{\vartheta}(F) \right| \leq K_{\vartheta} q^n_{\vartheta} \gamma_{\vartheta}(F).
\]
Since each \( A \in \mathcal{B}_1^k \) is a countable union of disjoint \( C_k \), we obtain (\ref{4.004}), and thus the proof is complete.
\end{proof}
Define 
\[
L_N := \max_{1 \leq n \leq N} b_n(x), \quad x \in \Omega.
\]
In the sequel, we discuss the asymptotic behavior of the largest digit \( L_N \).

\begin{theorem} \label{Th.4.4}
For any \( y > 0 \), we have
\begin{equation} \label{4.0700}
\lim_{N \to \infty} \gamma_{\vartheta}\left( x \in \Omega: L_N(x) < \frac{N y}{\log(1+\vartheta^2)} \right) = \exp\left( -\frac{1}{y}\right).
\end{equation}
\end{theorem}

\begin{proof}
\textit{1st step.} Let
\[
A_n = \{x \in \Omega: b_n(x) \geq w\},
\]
which implies
\[
\bigcap_{n=1}^{N} A_n^C = \{x \in \Omega: L_N(x) < w\} =: B_N.
\]
Since \( B_N \) represents the event where none of the \( A_n \) occurs, the Poincaré identity gives
\begin{equation}
\gamma_{\vartheta} (B_N) = \sum_{k=0}^{N} (-1)^k S_k \label{4.007}
\end{equation}
with
\[
S_0 = 1, \quad S_k = \sum_{1 \leq n_1 < n_2 < \ldots < n_k \leq N} \gamma_{\vartheta} \left( A_{n_1} \cap \ldots \cap A_{n_k} \right).
\]
Thus, equation (\ref{4.007}) provides an expression for the distribution function of \( L_N \). By choosing \( w = \left\lfloor \frac{N y}{\log(1+\vartheta^2)} \right\rfloor \), we show that the tail \( \sum_{k \geq Z} S_k \), where \( Z \) is a sufficiently large but fixed value, can be made arbitrarily small.

By repeatedly applying equation (\ref{4.004}) and referring to Lemma \ref{lema4.2}, we obtain that
\begin{equation} \label{4.008} 
\gamma_{\vartheta} \left( A_{n_1} \cap \ldots \cap A_{n_k}\right) \leq (1+K_{\vartheta})^{k-1} \gamma_{\vartheta}(A_{n_1})\gamma_{\vartheta}(A_{n_2}) \cdots \gamma_{\vartheta}(A_{n_k}) < (1+K_{\vartheta})^k p_{\vartheta}^k(w).
\end{equation}
For sufficiently large values of \( N \), we obtain
\[
w = \left\lfloor \frac{N y}{\log(1+\vartheta^2)} \right\rfloor \geq \frac{1}{2} \frac{N y}{\log(1+\vartheta^2)},
\]
hence
\[
p_{\vartheta}(w) \leq \frac{1}{w \log(1+\vartheta^2)} \leq \frac{2}{Ny}.
\]
Therefore
\[
\sum_{k \geq Z} S_k < \sum_{k \geq Z} \frac{N!}{(N-k)!k!} (1+K_{\vartheta})^k p_{\vartheta}^k(w) \leq \sum_{k \geq Z} \frac{N!}{(N-k)!k!} N^{-k} \left( \frac{2(1+K_{\vartheta})}{y} \right)^k
\]
\begin{equation} \label{4.0800} 
\leq \sum_{k \geq Z} \frac{1}{k!} \left( \frac{2(1+K_{\vartheta})}{y} \right)^k < \frac{1}{Z!} \left( \frac{4K_{\vartheta}}{y} \right)^Z \exp\left( \frac{4K_{\vartheta}}{y} \right) \leq \varepsilon
\end{equation}
as the value of \( Z \) is increased sufficiently.

\textit{2nd step.} Let us split \( S_k \) into two terms when considering \( k < Z \):
\begin{equation} \label{4.009}
S_k = S_k^* + R_k,
\end{equation}
where \( S_k^* \) represents the sum over all \( n_1 < n_2 < \ldots < n_k \) with \( n_{i+1} - n_i \geq t \) (\( i \geq 1 \)), where \( t \) is a positive integer determined as follows. Let \( \eta > 0 \) be an arbitrary real number, and let \( t \) be the smallest integer \( n \) such that \( K_{\vartheta} q_{\vartheta}^n < \eta \). Next, we estimate \( S_k^* \). Using repeated applications of (\ref{4.004}) and another reference to Lemma \ref{lema4.2}, we find that for any term in \( S_k^* \),
\[
\gamma_{\vartheta} \left( A_{n_1} \cap \ldots \cap A_{n_k} \right) = p_{\vartheta}^k(w) \left( 1 + \mathcal{O}_k(\eta) \right), \quad n_i + t \leq n_{i+1}.
\]
Thus, we obtain
\begin{equation} \label{4.0010} 
S_k^* = \frac{(N-(t-1)(k-1))!}{(N-(t-1)(k-1)-k)!k!} p_{\vartheta}^k(w) \left( 1 + \mathcal{O}_k(\eta) \right).
\end{equation}
To estimate \( R_k \) in (\ref{4.009}), observe that the overall estimation (\ref{4.008}) applies to each of its individual terms, and the number of terms is
\[
\frac{N!}{(N-k)!k!} - \frac{(N-(t-1)(k-1))!}{(N-(t-1)(k-1)-k)!k!} = o(N^k).
\]
Hence, we have
\begin{equation} \label{4.0011} 
R_k = o(N^k p_{\vartheta}^k(w)).
\end{equation}
Considering that
\[
p_{\vartheta}(w) = p_{\vartheta} \left( \left\lfloor \frac{N y}{\log(1+\vartheta^2)} \right\rfloor \right) = \left( 1 + \mathcal{O}(N^{-1}) \right) \frac{1}{Ny},
\]
from (\ref{4.009}), (\ref{4.0010}), and (\ref{4.0011}), we deduce that
\begin{equation} \label{4.0130}
S_k = \left( 1 + \mathcal{O}_k(\eta) \right) \frac{y^{-k}}{k!} + o_N(1),
\end{equation}
where \( k \) is fixed, and \( o_N(1) \to 0 \) as \( N \to \infty \).

\textit{3rd step.} Finally, by (\ref{4.007}), (\ref{4.0800}), and (\ref{4.0130}), we establish that for any given positive integer \( Z \),
\[
\gamma_{\vartheta} \left( L_N < \left\lfloor \frac{N y}{\log(1+\vartheta^2)} \right\rfloor \right) = \sum_{k=0}^{Z-1} (-1)^k \left( 1 + \mathcal{O}_k(\eta) \right) \frac{y^{-k}}{k!} + o_N(1) + o_Z(1),
\]
where the last term approaches 0 as \( Z \to \infty \). Letting \( N \to \infty \), and then \( \eta \to 0 \), we deduce that for any positive integer \( Z \),
\[
\lim_{N \to \infty} \gamma_{\vartheta} \left( L_N < \left\lfloor \frac{N y}{\log(1+\vartheta^2)} \right\rfloor \right) = \sum_{k=0}^{Z-1} (-1)^k \frac{y^{-k}}{k!} + o_Z(1).
\]
Since the left-hand side remains independent of \( Z \), taking \( Z \to \infty \), we arrive at the limit relation (\ref{4.0700}), while considering that the argument \( w \) in \( \{L_N < w\} \) is an integer in \( \mathbb{N}_m \). Since
\[
\gamma_{\vartheta} \left( L_N < \left\lfloor \frac{N y}{\log(1+\vartheta^2)} \right\rfloor \right)
\leq
\gamma_{\vartheta} \left( L_N < \frac{N y}{\log(1+\vartheta^2)} \right)
\leq
\gamma_{\vartheta} \left( L_N < \left\lfloor \frac{N y}{\log(1+\vartheta^2)} \right\rfloor + 1 \right),
\]
the proof is complete.
\end{proof}

\begin{theorem} \label{Th.4.5}
For any $0<\delta <1$ and $y>0$, we have
\begin{equation} \label{4.11}
\gamma_{\vartheta}\left( x \in \Omega: L_N(x) < \frac{N y}{\log\left(1+\vartheta^2 \right) } \right) = \exp\left( -\frac{1}{y}\right) +\mathcal{O}\left( \exp\left( -(\log N)\right)^{\delta} \right),
\end{equation}
where the constant involved in $\mathcal{O}$ depends exclusively on $\delta$.
\end{theorem}
\begin{proof}
We follow the proof of Theorem \ref{Th.4.4} with a particular choice of $Z$ and $t$. 
We choose $Z = \left \lfloor \frac{\log N}{\log \log N}\right \rfloor$. 
For a specific $0<\delta <1$, we choose $\delta < \delta' < 1$, $\varepsilon >0$ and $\zeta > 0$ so that $1 - \delta' > \varepsilon + \zeta$. 
We assume that $y \geq (\log N)^{-\delta}$. 
Applying Stirling's formula, we derive
\[
\frac{1}{Z!} ~ \frac{1}{\sqrt{2\pi} \, Z^{Z+1/2} \exp(-Z)} \asymp \frac{\exp\left(\frac{\log N}{\log \log N} \right)}{\left( \frac{\log N}{\log \log N} \right)^{\frac{\log N}{\log \log N} +\frac{1}{2}}  }.
\]
For $N$ sufficiently large,
\[
e \leq \left( \frac{\log N}{\log \log N} \right)^{\zeta}.
\]
Thus, we obtain
\[
\exp\left(\frac{\log N}{\log \log N} \right) \leq \left( \left( \frac{\log N}{\log \log N} \right)^{\zeta} \right)^{\frac{\log N}{\log \log N}} < N^{\zeta}.
\]
Furthermore, for $N$ sufficiently large
\[
\left( \frac{\log N}{\log \log N} \right)^{\frac{\log N}{\log \log N}} > N^{1-\varepsilon}.
\]
As a result, we obtain
\begin{equation} \label{4.12}
\frac{1}{Z!} \ll \frac{1}{N^{1-\varepsilon-\zeta}}.
\end{equation}
Alternatively, it is obvious that
\[
\left(\frac{4K_{\vartheta}}{y}\right)^Z \exp\left(\frac{4K_{\vartheta}}{y}\right) \leq \left( 4K_{\vartheta}\left( \log N \right)^{\delta} \right)^{\frac{\log N}{\log \log N}} \exp \left( 4K_{\vartheta}\left( \log N \right)^{\delta} \right) < N^{\delta'}
\]
when $N$ is sufficiently large. 
Finally, using (\ref{4.0800}), we obtain
\begin{equation} \label{4.13}
\sum_{k \geq Z} S_k \ll N^{-a}
\end{equation}
for $0 < a < 1-\varepsilon - \zeta - \delta'$. 

Setting $t = \left \lfloor (\log N)^2\right\rfloor$, we estimate $R_k$ for $k<Z$
\begin{eqnarray*}
R_k &\leq& \left( \frac{N!}{(N-k)!k!} - \frac{(N-(t-1)(k-1))!}{(N-(t-1)(k-1)-k)!k!} \right) (1+K_{\vartheta})^k p^k_{\vartheta}(w) \\
&\ll& t \, Z\, N^{k-1} (2K_{\vartheta})^k p^k_{\vartheta}(w) \\
&\ll& (\log N)^3 \frac{1}{\log \log N} N^{k-1} (2K_{\vartheta})^k p^k_{\vartheta}(w) \\
&\ll& (\log N)^3 N^{k-1} \left( \frac{4K_{\vartheta}}{Ny}\right)^k.
\end{eqnarray*}
In cases where $\frac{4K_{\vartheta}}{y} > 1$, we proceed to evaluate, for $N$ sufficiently large,
\begin{equation} \label{4.14}
R_{k} \leq \frac{1}{N^{1-\varepsilon}}\left(\frac{4K_{\vartheta}}{y}\right)^Z 
\leq \frac{1}{N^{1-\varepsilon}} \left( 4K_{\vartheta} (\log N)^{\delta}\right)^{\frac{\log N}{\log \log N}} < N^{-a}
\end{equation}
for $0 < a < 1-\varepsilon - \delta$. 
When $\frac{4K_{\vartheta}}{y} < 1$, the estimation becomes relatively straightforward.
We can select the value of $a$ to be the same as that in equation (\ref{4.13}). 

As a result, the number of terms in $S^*_k$, $k < Z$, is given by  
\[
\frac{N!}{(N-k)!k!} + \mathcal{O}\left( N^{k-1}(\log N)^3\right).
\]
We have
\begin{eqnarray} \label{4.15}
S^*_k &=& \left( \frac{N!}{(N-k)!k!} + \mathcal{O}\left( N^{k-1}(\log N)^3\right) \right) \left( 1+ \mathcal{O}\left( N^{-1}\right)  \right)^k \left(Ny \right)^{-k} \left(1+\beta K_{\vartheta}q^{(\log N)^2}_{\vartheta} \right)^k \nonumber \\
&=& \frac{y^{-k}}{k!} + \mathcal{O}\left( N^{-a} \right).
\end{eqnarray}
where $|\beta| \leq 1$. 
Subsequently, using (\ref{4.14}) and (\ref{4.15}), we deduce that 
$S_k = \frac{y^{-k}}{k!} + \mathcal{O}\left( N^{-a} \right)$. 
In conclusion, using (\ref{4.13}), we obtain
\[
\gamma_{\vartheta} (B_N) = \sum_{k=0}^{Z-1} \left( (-1)^k \frac{y^{-k}}{k!} + \mathcal{O}\left( N^{-a} \right)  \right) + \mathcal{O}\left( N^{-a} \right) = 
\sum_{k=0}^{Z-1} (-1)^k \frac{y^{-k}}{k!} + \mathcal{O}\left( N^{-a'} \right) = \exp\left(-\frac{1}{y} \right) 
+ \mathcal{O}\left( N^{-a'} \right),
\] 
with $0 < a' < a$. 
\end{proof}

\section{Some iterated logarithm results}

We begin with the following quantitative Borel-Cantelli lemma.

\begin{lemma} \cite{Ph-1967} \label{lema5.1}
Let $\{E_N\}_{n \geq 1}$ be a sequence of measurable sets in a probability space $(X, \mathcal{X}, \mu)$. Denote by $A(N,x)$ the number of integers $n \leq N$ such that $x \in E_n$, i.e., 
\[
A(N,x) = \sum_{n \leq N} \chi_{E_n}(x),
\]
where $\chi_{E_n}$ is the characteristic function of $E_n$. Define
\[
\varphi(N) := \sum_{n \leq N} \mu(E_n).
\]
Suppose there exists a convergent series $\sum_{k \geq 1} c_k$ with $c_k \geq 0$ such that for all integers $n > \ell$, we have 
\begin{equation} \label{5.1}
\mu(E_n \cap E_{\ell}) \leq \mu(E_n) \mu(E_{\ell}) + \mu(E_n) c_{n-\ell}.
\end{equation}
Then for any $\varepsilon > 0$, 
\begin{equation} \label{5.2}
A(N,x) = \varphi(N) + \mathcal{O}\left(\varphi^{1/2}(N) \log^{3/2 + \varepsilon} \varphi(N)\right) \quad \mu\text{-a.s.}
\end{equation}
\end{lemma}

\begin{theorem} \label{th.5.2}
For almost every $x \in [0, \vartheta]$, we have
\[
\liminf_{N \to \infty} \frac{L_N(x) \log \log N}{N} = \frac{1}{\log \left( 1+ \vartheta^2\right)}.
\]
\end{theorem}

\begin{proof}
Since for all $A \in \mathcal{B}_{[0, \vartheta]}$, we have
\[
\frac{\lambda_{\vartheta}(A)}{\left( 1+ \vartheta^2\right)\log \left( 1+ \vartheta^2\right)} \leq \gamma_{\vartheta}(A) \leq \frac{\lambda_{\vartheta}(A)}{\log \left( 1+ \vartheta^2\right)},
\]
the measures $\gamma_{\vartheta}$ and $\lambda_{\vartheta}$ are equivalent. Therefore, we proceed to prove for all $x$ except a set of $\gamma_{\vartheta}$-measure $0$. Consider integers $M$ and $N$ with $M, N \geq 0$. Define
\[
L(M, N, x) := \max_{M < n \leq M + N} b_n(x),
\]
\[
\varphi(n) := \frac{n}{\log \log n \log\left( 1+ \vartheta^2\right)},
\]
and
\[
E_k := \left( x \in \Omega : L\left( k^{2k}, k^{2(k+1)}, x\right) \leq \varphi \left(k^{2(k+1)} \right) \right).
\]
Due to the $T_{\vartheta}$-invariance of $\gamma_{\vartheta}$, we can deduce from Theorem \ref{Th.4.5} that for any integer $k \geq k_0$, 
\begin{eqnarray} \label{5.3}
\gamma_{\vartheta}(E_k) &=& \gamma_{\vartheta} \left( x \in \Omega : L\left( k^{2k}, k^{2(k+1)}, x \right) \leq \varphi \left(k^{2(k+1)} \right) \right) \nonumber \\
&=& \gamma_{\vartheta} \left( x \in \Omega : L\left( 0, k^{2(k+1)}, x \right) \leq \varphi \left(k^{2(k+1)} \right) \right) \nonumber \\
&\geq& \frac{1}{2} \exp \left( -\log \log k^{2(k+1)}\right) \geq \frac{1}{8} (k \log k)^{-1}.
\end{eqnarray}
Clearly, $E_k$ depends only on $b_n(x)$ with $k^{2k} < n \leq k^{2(k+1)} + k^{2k}$. Consequently, using Lemma \ref{lema4.3}, for any pair of integers $k < \ell$, we have
\begin{equation*}
 \left| \gamma_{\vartheta}(E_k \cap E_{\ell}) - \gamma_{\vartheta}(E_k) \gamma_{\vartheta}(E_{\ell})\right| \leq K_{\vartheta}q^{\ell-k}_{\vartheta} \gamma_{\vartheta}(E_k) \gamma_{\vartheta}(E_{\ell}),
\end{equation*}
since $(k+1)^{2(k+1)} - k^{2(k+1)} - k^{2k} \geq 1$. 

From Lemma \ref{lema5.1}, we can conclude that $x \in E_k$ for infinitely many $k$ (almost surely), provided that $\varphi (N) \gg \log \log N$ according to (\ref{5.3}). By Lemma \ref{lema4.2}, we have
\begin{eqnarray*}
\gamma_{\vartheta}(F_k) &:=& \gamma_{\vartheta} \left( x \in \Omega : L\left( 0, k^{2k}, x \right) \geq \varphi \left(k^{2(k+1)} \right) \right) \\
&\leq& \sum_{n \leq k^{2k}} \gamma_{\vartheta} \left(x \in \Omega : b_n(x) \geq \varphi\left(k^{2(k+1)} \right)  \right) = k^{2k} p_{\vartheta} \left( \varphi \left( k^{2(k+1)}\right) \right)  \\
&\leq& k^{2k} \log \log k^{2(k+1)} \cdot k^{-2(k+1)} \leq k^{-3/2}.
\end{eqnarray*}
Hence, by Lemma \ref{lema5.1}, $x \in F_k$ only for finitely many $k$ (almost surely). Thus, 
\[
x \in E_k \setminus F_k = \left( x \in \Omega : L\left( 0, k^{2k} + k^{2(k+1)}, x\right) \leq \varphi \left(k^{2(k+1)} \right)  \right)
\]
for infinitely many $k$ (almost surely), which implies that 
\[
L\left( 0, k^{2(k+1)}, x \right) \leq \varphi \left(k^{2(k+1)} \right)
\]
holds for infinitely many $k$ (almost surely). Therefore, 
\begin{equation} \label{5.4}
\liminf_{N \to \infty} \frac{L_N(x) \log \log N}{N} \leq \frac{1}{\log \left( 1+ \vartheta^2\right)} \quad \text{a.e.}
\end{equation}
Now, we prove the converse inequality. Let $b > 1$. Again, by Theorem \ref{Th.4.5},
\begin{eqnarray*}
\gamma_{\vartheta}(G_k) &:=& \gamma_{\vartheta} \left( x \in \Omega : L\left( 0, \left\lfloor b^k \right\rfloor, x\right) \leq b^{-2} \varphi \left( \left\lfloor b^{k+1} \right\rfloor \right) \right) \\
&\ll& \exp \left( -b \log \log b^k \right) \ll k^{-b}.
\end{eqnarray*}
By Lemma \ref{lema5.1}, since $\sum k^{-b} < \infty$, it follows that $x \in G_k$ only for finitely many $k$ (almost surely). This means that
\[
L\left( 0, \left\lfloor b^k \right\rfloor, x \right) > b^{-2} \varphi \left( \left\lfloor b^{k+1} \right\rfloor \right)
\]
holds for all $k \geq k_0 (x,b)$. For a given value of $N$ such that $\left\lfloor b^k \right\rfloor \leq N < b^{k+1}$ where $k \geq k_0 (x,b)$, since
$L\left( 0, \left\lfloor b^k \right\rfloor, x \right) \leq L_N(x)$ 
and 
$\varphi (N) \leq \varphi \left( \left\lfloor b^{k+1} \right\rfloor \right)$, we conclude that 
\[
L_N(x) > b^{-2} \varphi(N) \quad \text{a.e. } x.
\]
Since this holds for any $b > 1$, we obtain 
\[
\liminf_{N \to \infty} \frac{L_N(x) \log \log N}{N} \geq \frac{1}{\log \left( 1+ \vartheta^2\right)} \quad \text{a.e.}
\]
By (\ref{5.4}), the proof is completed.
\end{proof}

There is no analogous result for Theorem \ref{th.5.2} with a finite nonzero superior limit. This follows from the following theorem.

\begin{theorem} \label{th.5.3}
Let $\{\varphi(n)\}_{n \geq 1}$ be a positive nondecreasing sequence. Then for a.e. $x \in [0, \vartheta]$, 
\begin{equation} \label{5.5}
L_N(x) > \varphi (N)
\end{equation}
has finitely many or infinitely many solutions in integers $N$ according as the series  
\begin{equation} \label{5.6}
\sum_{n \geq 1} \frac{1}{\varphi(n)}
\end{equation}
converges or diverges.
\end{theorem}

\begin{proof}
Indeed, if $\sup \varphi(n) < \infty$, then the divergence of (\ref{5.5}) implies, according to Theorem \ref{th.B-B}, that $b_n(x) > \varphi(n)$ holds for infinitely many $n$ (a.e. $x$).

On the other hand, when $\varphi(n) \nearrow \infty$, the behavior of (\ref{5.5}) is determined by whether the inequality $b_n(x) > \varphi(n)$ holds finitely or infinitely often. This, in turn, leads to the conclusion that, by Theorem \ref{th.B-B}, this behavior holds for a.e. $x$ based on whether the series (\ref{5.6}) converges or diverges.
\end{proof}

\begin{corollary} \label{cor.5.4}
Let $\{\varphi(n)\}_{n \geq 1}$ be as in Theorem \ref{th.5.3}. Then for a.e. $x \in [0, \vartheta]$,
\begin{equation} \label{5.7}
\limsup_{N \to \infty} \frac{L_N(x)}{\varphi(N)}
\end{equation}
is either $0$ or $\infty$.
\end{corollary}

\begin{proof}
We distinguish the cases where the series (\ref{5.6}) converges or diverges. If the series (\ref{5.6}) converges, we choose a monotone sequence $\{\alpha_n\}_{n \geq 1}$ tending to $\infty$ but so slowly that still $\sum_{n \geq 1} \frac{\alpha_n}{\varphi(n)} < \infty$. Therefore, according to Theorem \ref{th.5.3}, the inequality $L_N(x) > \frac{\varphi(N)}{\alpha_N}$ holds only for finitely many $N$ (a.e. $x$). Hence, (\ref{5.7}) vanishes for a.e. $x$. 

If the series (\ref{5.6}) diverges, we consider a monotone sequence $\{\alpha_n\}_{n \geq 1}$ tending to $0$ such that $\sum_{n \geq 1} \frac{\alpha_n}{\varphi(n)} = \infty$. Hence, $L_N(x) > \frac{\varphi(N)}{\alpha_N}$ holds for infinitely many $N$ (a.e. $x$), and thus (\ref{5.7}) is infinite for a.e. $x$.
\end{proof}

%\noindent \textbf{Acknowledgments}

%The authors would like to thank the referee for carefully reading our manuscript and for giving such constructive comments which substantially helped improving the quality of the paper.

\end{document}